\documentclass[a4paper,12pt]{amsart}
\usepackage{amsmath}
 \usepackage{latexsym, amsthm, amscd, euscript, float, subfig}
\usepackage{amssymb}
\usepackage{ifthen}
\usepackage{graphicx}
\usepackage{float}
\usepackage{cite}
\usepackage{amsfonts}
\usepackage{amscd}
\usepackage{amsxtra}
\usepackage{color}

\newtheorem{theorem}{Theorem}[section]
\newtheorem{lemma}[theorem]{Lemma}

\theoremstyle{definition}

\numberwithin{equation}{section}

\def\be{\begin{equation}}
\def\ee{\end{equation}}

\newcounter{alphabet}


\begin{document}
\bibliographystyle{amsplain}
\title[Second Hankel determinant of logarithmic inverse coefficients]{Second Hankel Determinant for Logarithmic Inverse Coefficients of Convex and Starlike Functions}
\author[Vasudevarao Allu]{Vasudevarao Allu}
\address{Vasudevarao Allu, School of Basic Sciences, Indian Institute of Technology Bhubaneswar,
Bhubaneswar-752050, Odisha, India.}
\email{avrao@iitbbs.ac.in}
\author[Amal Shaji]{Amal Shaji}
\address{Amal Shaji, School of Basic Sciences, Indian Institute of Technology Bhubaneswar,
Bhubaneswar-752050, Odisha, India.}
\email{amalmulloor@gmail.com}
\subjclass[2010]{30C45, 30C50, 30C55.}
\keywords{Univalent functions, Logarithmic coefficients, Hankel determinant, Starlike and Convex functions, Schwarz function.}

\begin{abstract}
In this paper, we obtain the sharp bounds of the second Hankel determinant of logarithmic  inverse coefficients  for the starlike and convex functions.
\end{abstract}

\maketitle

\section{Introduction}\label{Introduction}
Let $\mathcal{H}$ denote the class of analytic functions in the unit disk $\mathbb{D}:=\{z\in\mathbb{C}:\, |z|<1\}$. Here $\mathcal{H}$ is 
a locally convex topological vector space endowed with the topology of uniform convergence over compact subsets of $\mathbb{D}$. Let $\mathcal{A}$ denote the class of functions $f\in \mathcal{H}$ such that $f(0)=0$ and $f'(0)=1$.  Let $\mathcal{S}$ 
denote the subclass of  $\mathcal{A}$ consisting of functions which are univalent ({\em i.e., one-to-one}) in $\mathbb{D}$. 
If $f\in\mathcal{S}$ then it has the following series representation
\begin{equation}\label{f}
	f(z)= z+\sum_{n=2}^{\infty}a_n z^n, \quad z\in \mathbb{D}.
\end{equation}

For $q,n \in \mathbb{N}$, the Hankel determinant $H_{q,n}(f)$ of Taylor’s coefficients of function $f \in \mathcal{A}$ of the form \eqref{f} is defined by
$$
H_{q,n}(f) =
\begin{vmatrix}
	a_n & a_{n+1}  & \cdots & a_{n+q-1} \\ 
	a_{n+1} & a_{n+2}  & \cdots & a_{n+q} \\
	\vdots & \vdots &  \ddots &\vdots \\
	a_{n+q-1} & a_{n+q}  & \cdots & a_{n+2(q-1)}
\end{vmatrix}.
$$
The Hankel determinant for various order is also studied recently by several authors in different contexts;
for instance see \cite{Pom66,Pom67,ALT}.
One can easily observe that the Fekete-Szeg\"{o} functional is the second Hankel determinant $H_{2,1}(f)$. Fekete-Szeg$\ddot{o}$  then
further generalized the estimate $|a_3 - \mu a_2
^2|$ with $\mu$ real for $f$ given by \eqref{f} (see \cite[Theorem 3.8]{Duren-book-1983}). \\[2mm]
Let $g$ be the inverse function of $f\in \mathcal{S}$  defined in a neighborhood of the
origin with the Taylor series expansion
\begin{equation}\label{inverse}
g(w)=f^{-1}(w)=w+\sum_{n=2}^{\infty}A_n w^n,
\end{equation}
where we may choose $|w| < 1/4$, as we know from Koebe’s $1/4$-theorem. Using
 variational method, L\"{o}wner \cite{Lowner} obtained the sharp estimate:
$$
|A_n| \leq K_n  \quad \text{for each}\,\,\, n \in \mathbb{N}$$
where $K_n = (2n)!/(n!(n + 1)!)$ and $K(w) = w + K_2w_2 + K_3w_3 + \cdots $ is the
inverse of the Koebe function. There has been a good deal of interest in determining the behavior of the inverse coefficients of $f$ given in \eqref{f} when the
corresponding function $f$ is restricted to some proper geometric subclasses
of $\mathcal{S}$.\\[2mm]
Let $f(z)= z+\sum_{n=2}^{\infty}a_n z^n $ be a function in  class $\mathcal{S}$. Since $f(f^{-1})(w)=w$ and using \eqref{inverse}, it follows that
\begin{equation}\label{inversetaylor}
\begin{aligned}
& A_{2}=-a_2,\\[2mm]
& A_{3}=-a_3+2a_2^2,\\[2mm]
& A_{4}=-a_4+5a_2a_3-5a_2^3.
\end{aligned}
\end{equation}
\\[2mm]
The {\it Logarithmic coefficients} $\gamma_{n}$ of $f\in \mathcal{S}$ are defined by,
\begin{equation}\label{amal-1}
	F_{f}(z):= \log\frac{f(z)}{z}=2\sum\limits_{n=1}^{\infty}\gamma_{n}z^{n}, \quad z \in \mathbb{D}.
\end{equation}
The logarithmic coefficients $\gamma_{n}$ play a central role in the theory of univalent functions. A very few exact upper bounds for $\gamma_{n}$ seem to have been established. The significance of this problem in the context of Bieberbach conjecture was pointed by Milin\cite{milin} in his conjecture. Milin \cite{milin} has conjectured that for $f\in \mathcal{S}$ and $n\ge 2$, 
$$\sum\limits_{m=1}^{n}\sum\limits_{k=1}^{m}\left(k|\gamma_{k}|^{2}-\frac{1}{k}\right)\le 0,$$
which led De Branges, by proving this conjecture, to the proof of Bieberbach conjecture  \cite{De Branges-1985}. For the Koebe function $k(z)=z/(1-z)^{2}$, the logarithmic coefficients are $\gamma_{n}=1/n$. Since the Koebe function $k$ plays the role of extremal function for most of the extremal problems in the class $\mathcal{S}$, it is expected that $|\gamma_{n}|\le1/n$ holds for functions in $\mathcal{S}$. But this is not true in general, even in order of magnitude. Indeed, there exists a bounded function $f$ in the class $\mathcal{S}$ with logarithmic coefficients $\gamma_{n}\ne O(n^{-0.83})$ (see \cite[Theorem 8.4]{Duren-book-1983}). By differentiating \eqref{amal-1} and  the equating coefficients we obtain
\begin{equation}\label{gamma}
\begin{aligned}
& \gamma_{1}=\frac{1}{2}a_{2}, \\[2mm]
& \gamma_{2}=\frac{1}{2}(a_{3}-\frac{1}{2}a_{2}^{2}),\\[2mm]
& \gamma_{3}=\frac{1}{2}(a_{4}-a_{2}a_{3}+\frac{1}{3}a_{2}^{3}).
\end{aligned}
\end{equation}
If $f\in \mathcal{S}$, it is easy to see that $|\gamma_{1}|\le 1$, because $|a_2| \leq 2$. Using the Fekete-Szeg$\ddot{o}$ inequality \cite[Theorem 3.8]{Duren-book-1983} for functions in $\mathcal{S}$ in (1.4), we obtain the sharp estimate 
$$|\gamma_{2}|\le\frac{1}{2}\left(1+2e^{-2}\right)=0.635\ldots.$$
For $n\ge 3$, the problem seems much harder, and no significant bound for $|\gamma_{n}|$ when $f\in \mathcal{S}$ appear to be known. In 2017, Ali and Allu\cite{vasu2017} obtained the initial logarithmic coefficients bounds for close-to-convex functions. The problem of computing the bound of the logarithmic coefficients is also considered in \cite{cho,PSW20,vasu-2018,Thomas-2016} for several subclasses of close-to-convex functions.
\\[2mm]
The notion of logarithmic inverse coefficients, {\em i.e.}, logartithmic coefficients of inverse of $f$, was proposed by ponnusamy {\it et al.} \cite{samyinverselog}. The {\it logarithmic
inverse coefficients} $\Gamma_n$, $n \in \mathbb{N}$, of $f$ are defined by the equation
\begin{equation}\label{Gamma}
	F_{f^{-1}}(w):= \log\frac{f^{-1}(w)}{w}=2\sum\limits_{n=1}^{\infty}\Gamma_{n}w^{n}, \quad |w|<1/4.
\end{equation}
In \cite{samyinverselog} Ponnusamy {\it et al.} found the sharp upper bound  for the logarithmic inverse coefficients for the class $\mathcal{S}$. In fact ponnusamy {\it et al.} \cite{samyinverselog} proved that when $f\in \mathcal{S}$, $$|\Gamma_n| \leq \frac{1}{2n}
  \left(
  \begin{matrix}
    2n \\
    n
  \end{matrix}
  \right),
  \quad n \in \mathbb{N}
$$
and equality holds only for the Koebe function or one of its rotations. Further, ponnusamy {\it et al.} \cite{samyinverselog} obtained sharp bound for the initial logarithmic inverse coefficients for some of the important geometric subclasses of $
\mathcal{S}$.\\

Recently, Kowalczyk and Lecko \cite{adam} together have proposed the study of the Hankel determinant whose entries are logarithmic coefficients of $ f \in \mathcal{S}$, which is given by

$$
H_{q,n}(F_f/2) =
\begin{vmatrix}
	\gamma_n & \gamma_{n+1}  & \cdots & \gamma_{n+q-1} \\ 
	\gamma_{n+1} & \gamma_{n+2}  & \cdots & \gamma_{n+q} \\
	\vdots & \vdots &  \ddots &\vdots \\
	\gamma_{n+q-1} & \gamma_{n+q}  & \cdots & \gamma_{n+2(q-1)}
\end{vmatrix}.
$$
Kowalczyk and Lecko \cite{adam} have obtained the sharp bound of the second Hankel determinant of $F_f/2$, {\em i.e.,} $H_{2,1}(F_f/2)$ for starlike and convex functions. The problem of computing the sharp bounds of $H_{2,1}(F_f/2)$ has been considered by many authors for various subclasses of $\mathcal{S}$ (See \cite{vibhuthi,vibhuthi2,adam2,Mundalia}). 
\\

In this paper, we consider the notion of the second Hankel determinant for logarithmic inverse coefficients.
Let $ f \in \mathcal{S}$ given by \eqref{f}, then the second Hankel determinant of $F_{f^{-1}}/2$ by using \eqref{gamma}, is given by

\begin{equation}\label{hankel}
\begin{aligned}
	H_{2,1}(F_{f^{-1}}/2)
&=\Gamma_1\Gamma_3-\Gamma_{2}^2 \\[2mm]
&=\frac{1}{4}\left(A_2A_4-A_3^2+\frac{1}{4}A_2^4\right)\\[2mm]
&=\frac{1}{48}\left(13a_2^4-12a_2^2a3-12a_3^2+12a_2a_4\right).
\end{aligned}
\end{equation}
\\[1mm]
It is now appropriate to remark that $H_{2,1}(F_{f^{-1}}/2)$ is invariant under rotation, since for $f_{\theta}(z):=e^{-i \theta} f\left(e^{i \theta} z\right), \theta \in \mathbb{R}$ when $f \in \mathcal{S}$ we have
$$
H_{2,1}(F_{f_{\theta}^{-1}}/2)=\frac{e^{4 i \theta}}{48}\left(13a_2^4-12a_2^2a3-12a_3^2+12a_2a_4\right)=e^{4 i \theta} H_{2,1}(F_{f^{-1}}/2) .
$$
The main aim of this paper is to find sharp upperbound for $|H_{2,1}(F_{f^{-1}}/2)|
$ when $f$ belongs to the class of convex or starlike functions. A domain $\Omega \subseteq \mathbb{C}$ is said to be starlike domain with respect to a point $z_{0}\in \Omega$ if the line segment joining $z_{0}$ to any point in $\Omega$ lies entirely in $\Omega$. If $z_0$ is the origin then we say that $\Omega$ is a starlike domain. A function $f \in \mathcal{A}$ is said to be starlike function if $f(\mathbb{D})$ is a starlike domain. We denote by $\mathcal{S}^*$ the class of starlike functions $f$ in $\mathcal{S}$. It is well-known that a function $f \in \mathcal{A}$ is in $\mathcal{S}^*$ if, and only if, 
\begin{equation}\label{star}
{\rm Re\,}\left( \cfrac{zf'(z)}{f(z)} \right) > 0 \quad \text{for}\,\, z \in \mathbb{D}.
\end{equation}
Further, a domain $\Omega\subseteq\mathbb{C}$ is called convex if the line segment joining any two points of $\Omega$ lies entirely in $\Omega$. A function $f\in\mathcal{A}$ is called convex if $f(\mathbb{D})$ is a convex domain. We denote $\mathcal{C}$  the class of convex functions in $\mathcal{S}$. A function $f \in \mathcal{A}$ is in $\mathcal{C}$ if, and only if, 
\begin{equation}\label{convex}
{\rm Re\,}\left( 1+\cfrac{zf''(z)}{f'(z)} \right) > 0 \quad \text{for}\,\, z \in \mathbb{D}.
\end{equation}
\section{Preliminary Results}
In this section, we present the key lemmas which will be used to prove the main results of this paper. Let $\mathcal{P}$ denote the class of all analytic functions $p$ having positive real part in $\mathbb{D}$, with the form
\begin{equation}\label{p}
p(z)=1+c_{1} z+c_{2} z^{2}+c_{3} z^{3}+ \cdots .
\end{equation}
A member of $\mathcal{P}$ is called a Carathéodory function. It is known that $\left|c_{n}\right| \leq 2, n \geq 1$ for a function $p \in \mathcal{P}$. By using \eqref{star} and \eqref{convex}, functions in the classes $\mathcal{S}^*$ and $\mathcal{C}$ can be represented interms of functions in Carathéodory class $\mathcal{P}$.
\\[2mm]

Parametric representations of the coefficients are often useful. In Lemma \ref{caratheodary}, the formula \eqref{c1} is due to Carathéodary \cite{Duren-book-1983}. The formula \eqref{c2} can be found in \cite{Pombook}. In 1982, Libera and Zlotkiewicz \cite{LZ1,LZ2} derived the formula \eqref{c3} with the assumption that $c_1 \geq 0$. Later, Cho {\it et al.} \cite{cholecko} have derived the formula \eqref{c3} in general case and they have also given the explicit form of extremal function.

\begin{lemma}\label{caratheodary}

  If $p \in \mathcal{P}$ is of the form \eqref{p}, then

\begin{equation}\label{c1}
 c_{1}=2 p_{1}
 \end{equation}
 \begin{equation}\label{c2}
 c_{2}=2 p_{1}^{2}+2\left(1-p_{1}^{2}\right) p_{2}
 \end{equation}
and
\begin{equation}\label{c3}
c_{3}=2 p_{1}^{3}+4\left(1-p_{1}^{2}\right) p_{1} p_{2}-2\left(1-p_{1}^{2}\right) p_{1} p_{2}^{2}+2\left(1-p_{1}^{2}\right)\left(1-\left|p_{2}\right|^{2}\right) p_{3}
\end{equation}

for some $p_{1}, p_{2}, p_{3} \in \overline{\mathbb{D}}:=\{z \in \mathbb{C}:|z| \leq 1\}$.
\\[2mm]
For $p_{1} \in \mathbb{T}:=\{z \in \mathbb{C}:|z|=1\}$, there is a unique function $p \in \mathcal{P}$ with $c_{1}$ as in \eqref{c1}, namely

$$
p(z)=\frac{1+p_{1} z}{1-p_{1} z}, \quad z \in \mathbb{D} .
$$
\\[2mm]
For $p_{1} \in \mathbb{D}$ and $p_{2} \in \mathbb{T}$, there is a unique function $p \in \mathcal{P}$ with $c_{1}$ and $c_{2}$ as in \eqref{c1} and \eqref{c2}, namely

\begin{equation}\label{second}
p(z)=\frac{1+\left(p_{1}+\overline{p_{1}} p_{2}\right) z+p_{2} z^{2}}{1-\left(p_{1}-\overline{p_{1}} p_{2}\right) z-p_{2} z^{2}} .
\end{equation}
\\[2mm]
For $p_{1}, p_{2} \in \mathbb{D}$ and $p_{3} \in \mathbb{T}$, there is unique function $p \in \mathcal{P}$ with $c_{1}, c_{2}$, and $c_{3}$ as in \eqref{c1}-\eqref{c2}, namely

$$
p(z)=\frac{1+\left(\overline{p_{2}} p_{3}+\overline{p_{1}} p_{2}+p_{1}\right) z+\left(\overline{p_{1}} p_{3}+p_{1} \overline{p_{2}} p_{3}+p_{2}\right) z^{2}+p_{3} z^{3}}{1+\left(\overline{p_{2}} p_{3}+\overline{p_{1}} p_{2}-p_{1}\right) z+\left(\overline{p_{1}} p_{3}-p_{1} \overline{p_{2}} p_{3}-p_{2}\right) z^{2}-p_{3} z^{3}}, \quad z \in \mathbb{D} .
$$
\end{lemma}
Next we recall the following well-known result due to Choi {\it et al.} \cite{choi}. Lemma $2.2$ plays an important role in the proof of our main results.

\begin{lemma}\label{y(a,b,c)}
Let $A, B, C$ be real numbers and

$$
Y(A, B, C):=\max _{z \in \overline{\mathbb{D}}}\left(\left|A+B z+C z^{2}\right|+1-|z|^{2}\right) .
$$

(i) If $A C \geq 0$, then

$$
Y(A, B, C)= \begin{cases}|A|+|B|+|C|, & |B| \geq 2(1-|C|), \\[2mm] 1+|A|+\cfrac{B^{2}}{4(1-|C|)}, & |B|<2(1-|C|) .\end{cases}
$$

(ii) If $A C<0$, then

$$
Y(A, B, C)= \begin{cases}1-|A|+\cfrac{B^{2}}{4(1-|C|)}, & -4 A C\left(C^{-2}-1\right) \leq B^{2} \wedge|B|<2(1-|C|), \\[2mm] 1+|A|+\cfrac{B^{2}}{4(1+|C|)}, & B^{2}<\min \left\{4(1+|C|)^{2},-4 A C\left(C^{-2}-1\right)\right\}, \\[2mm] R(A, B, C), & \text { otherwise, }\end{cases}
$$

where

$$
R(A, B, C)= \begin{cases}|A|+|B|+|C|, & |C|(|B|+4|A|) \leq|A B|, \\[2mm] -|A|+|B|+|C|, & |A B| \leq|C|(|B|-4|A|), \\[2mm] (|A|+|C|) \sqrt{1-\cfrac{B^{2}}{4 A C}}, & \text { otherwise. }\end{cases}
$$
\end{lemma}

\section{Main Results}\label{sec3}
Now we will prove the first main result of this paper. We obtain the following sharp bound for $H_{2,1}(F_{f^{-1}}/2)$ for functions in the class $\mathcal{C}$.
\begin{theorem}
Let $f\in\mathcal{C}$ given by \eqref{f} then
\begin{equation}\label{thm1}
	|H_{2,1}(F_{f^{-1}}/2)| \leq \frac{1}{33}.
\end{equation}
	The inequality is sharp.

\end{theorem}

\begin{proof}
Let $f\in \mathcal{C}$ be of the form \eqref{f}. Then by \eqref{convex},
\begin{equation}\label{3.1.1}
1+\cfrac{zf''(z)}{f'(z)}=p(z)
\end{equation}
for some $p \in \mathcal{P}$ of the form \eqref{p}. 
Since the class $\mathcal{C}$ is invaraint under rotation and the function is also rotationally invariant, we can assume that $c_1 \in [0,2]$.
By comparing the coefficients on both sides of \eqref{3.1.1} yields
	\begin{equation}\label{3.2.2}
		\begin{aligned}
			& a_2=\frac{1}{2}c_1, \\[2mm]
			& a_3=\frac{1}{6}(c_2+c_1^2), \\[2mm]
			& a_4=\frac{1}{24}\left(2c_3+3c_1c_2+c_1^3 \right).
		\end{aligned}
	\end{equation}
Hence by \eqref{hankel},
$$
H_{2,1}(F_{f^{-1}}/2)=\cfrac{1}{2304}\left(11c_1^4-20c_1^2c_2-16c_2^2+24c_1c_3\right).
$$
Now using \eqref{c1}-\eqref{c3} and by simplification, we obtain
\begin{equation}\label{3.1.3}
\begin{aligned}
H_{2,1}(F_{f^{-1}}/2)
&=\frac{p_1^4}{48}-\frac{1}{24}(1-p_1^2)p_1^2p_2-\frac{1}{72}(1-p_1^2)(2+p_1^2)p_2^2\\
& \quad \quad +\frac{1}{24}(1-p_1^2)(1-|p_1^2|)p_1p_3.
\end{aligned}
\end{equation}
Hereby we have the following cases on $p_1$.
\\[1mm]
\textbf{Case 1:} Let $p_1=1$, then by \eqref{3.1.3}, we get
$$
|H_{2,1}(F_{f^{-1}}/2)|=\frac{1}{48}.
$$
\textbf{Case 2:} Let $p_1=0$, then by \eqref{3.1.3}, we get
$$
|H_{2,1}(F_{f^{-1}}/2)|=\frac{1}{36}|p_2^2|\leq\frac{1}{36}.
$$
\textbf{Case 3:} Let $p_1 \in (0,1)$. Applying the triangle inequality in \eqref{3.1.3} and by using the fact
that $|p_3| \leq 1$, we obtain
\begin{equation}\label{3.1.4}
\begin{aligned}
H_{2,1}(F_{f^{-1}}/2)
&=\frac{1}{24}p_1(1-p_1^2)\left(\left| \frac{p_1^3}{2(1-p_1^2}-p_1 p_2-\frac{2+p_1^2}{3p_1}p_2^2 \right|+1-|p_2^2|\right)\\[2mm]
&\leq \frac{1}{24}p_1(1-p_1^2)\left(\left| A+B p_2+C p_2^2 \right|+1-|p_2^2|\right)
\end{aligned}
\end{equation}
where 
$$
A:=\cfrac{p_1^3}{2(1-p_1^2)}, \quad B:=-p_1,\quad C:=-\cfrac{2+p_1^2}{3p_1}.
$$
Since $AC < 0$, so we can apply case (ii) of Lemma \ref{y(a,b,c)}.\\
[3mm]
\textbf{3(a).} Note that for $p_{1} \in(0,1)$, we have
$$
-4 A C\left(\frac{1}{C^{2}}-1\right)-B^{2}=-\cfrac{p_{1}^{2}(14+p_1^2)}{3(2+p_{1}^{2})} \leq  0.
$$
Moreover, the inequality $|B|<2(1-|C|)$ is equivalent to $p_1(4 - 6 p_1 + 5 p_1^2) < 0$ which is not true for $p_{1} \in(0,1)$.\\
[3mm]
\textbf{3(b).} It is easy to check that
$$
\min \left\{4(1+|C|)^{2},-4 A C\left(\frac{1}{C^{2}}-1\right)\right\}=-4 A C\left(\frac{1}{C^{2}}-1\right),
$$
and from $3(a),$ we know that 
$$
-4 A C\left(\frac{1}{C^{2}}-1\right) \leq B^{2}.
$$
Therefore the inequality $B^{2} < \min \left\{4(1+|C|)^{2},-4 A C\left(\frac{1}{C^{2}}-1\right)\right\}$ does not holds for $0<p_1<1$.\\[2mm]
\textbf{3(c).} Note that the inequality $|C|(|B|+4|A|)-|A B| \leq 0$ is equivalent to $4+6p_1^2
-p_1^4\leq 0$, which is false for $p_{1} \in(0,1)$.\\
[2mm]
\textbf{3(d).}
Observe that the inequality 
$$|A B|-|C|(|B|-4|A|)=\cfrac{9 p_1^4+10 p_1^2-4}{1 - p1^2} \leq 0
$$
is equivalent to $9 p_1^4+10 p_1^2-4\leq 0$, which is true for 
$$
0<p_1 \leq p_1'=\frac{1}{3}\sqrt{\sqrt{61}-5}\approx 0.5588.
$$
It follows from Lemma \ref{y(a,b,c)} and the inequality \eqref{3.1.4} that,
\begin{equation}
\begin{aligned}
H_{2,1}(F_{f^{-1}}/2)
&\leq \frac{1}{24}p_1(1-p_1^2)(-|A|+|B|+|C|)\\
&= \frac{1}{144}(4+4p_1^2-11p_1^4)=h(p_1)
\end{aligned}
\end{equation}
where $h(x)=4+4x^2-11x^4, \quad 0<x\leq p_1'$.
By simple calculation we can show that maximum of the function $g(x)$ exists at the point $x_0=\sqrt{2/11}$. \\
Therefore we can conclude that, for $0<p_1\leq p_1',$ we have
$$
|H_{2,1}(F_{f^{-1}}/2)|\leq h\left(\sqrt{\frac{2}{11}}\right)=\frac{1}{33}.
$$
\textbf{3(e).} For $p_1'<p_1<1$, we use the last case of Lemma \ref{y(a,b,c)} together with \eqref{3.1.4} to obtain
\begin{equation}
\begin{aligned}
H_{2,1}(F_{f^{-1}}/2)
&\leq \frac{1}{24}p_1(1-p_1^2)(|C|+|A|) \sqrt{1-\frac{B^{2}}{4 A C}}\\[2mm]
&= \frac{1}{144}(p_1^4-2p_1^2+4)\sqrt{\cfrac{7-p_1^2}{4+2p_1^2}}=k(p_1)
\end{aligned}
\end{equation}
where $k(x)=\cfrac{1}{144}\left(x^4-2x^2+4\right)\sqrt{\cfrac{7-x^2}{4+2x^2}}, \quad p_1'<x<1$.\\[2mm]
Now we want to find the maximum of $k(x)$ over the interval $p_1'<x<1$. We observe that
$$
k'(x)=\cfrac{x}{144}  \sqrt{\cfrac{7 - x^2}{
 4 + 2 x^2}}\left( \frac{92 - 54x^2 - 15 x^4 + 4 x^6}{(-7 + x^2) (2 +
    x^2)}\right)=0,
$$\\[1mm]
    if, and only, if $92 - 54x^2 - 15 x^4 + 4 x^6=0$. However, all the real roots of the last equation lies outside the interval $p_1'<x<1$ and $k'(x)<0$ for $p_1'<x<1$. So $k$ is decreasing and hence $k(x)\leq k(p_1')$for $p_1'<x<1$.\\[1mm]
Therefore we can conclude that for $p_1'<x<1,$
we have
$$
 |H_{2,1}(F_{f^{-1}}/2)|\leq k(p_1') \approx 0.0290035.
$$
Summarizing parts from Case 1-3, it follows the desired inequality \eqref{thm1}.\\
By tracking back the above proof, we see that the equality in \eqref{thm1} holds when it is satisfied that 
\begin{equation}
p_1=\sqrt{\frac{2}{11}},\quad p_3=1,
\end{equation}
and
\begin{equation}\label{abc}
|A+Bp_2+Cp_2^2|+1-|p_2^2|=-|A|+|B|+|C|,
\end{equation}
where
$$
A=\cfrac{\sqrt{\frac{2}{11}}}{9},\,\, B=-\sqrt{\frac{2}{11}},\,\,C=4\sqrt{\frac{2}{11}}.
$$
Indeed we can easily verify that one of the solutions of equation \eqref{abc} is
$$
p_2=1.
$$
In view of Lemma \ref{y(a,b,c)}, we conclude that equality holds for the function $f\in \mathcal{A}$ given by \eqref{convex}, where the function $p \in \mathcal{P}$ of the form \eqref{second} with $p_1=\sqrt{2/11},p_2=1$ and $p_3=1$, that is 
$$
p(z)=\cfrac{1+2\sqrt{2/11}z+z^2}{1-z^2}.
$$
This complete the proof.
\end{proof}
Next, we obtained the following sharp bound for $H_{2,1}(F_{f^{-1}}/2)$ for functions in the class $\mathcal{S}^*$.
\begin{theorem}
Let $f\in\mathcal{S}$ given by \eqref{f} then
\begin{equation}\label{thm2}
	|H_{2,1}(F_{f^{-1}}/2)| \leq \frac{13}{12}.
\end{equation}
	The inequality is sharp.

\end{theorem}

\begin{proof}
Let $f\in \mathcal{C}$ be of the form \eqref{f}. Then by \eqref{star},
\begin{equation}\label{3.2.1}
\cfrac{zf'(z)}{f(z)}=p(z)
\end{equation}
for some $p \in \mathcal{P}$ of the form \eqref{p}. By comparing the coefficients on both the sides of \eqref{3.2.1}, we obtain
	\begin{equation}\label{3.2.2}
		\begin{aligned}
			& a_2=c_1, \\[2mm]
			& a_3=\frac{1}{2}(c_2+c_1^2), \\[2mm]
			& a_4=\frac{1}{6}\left(2c_3+3c_1c_2+c_1^3 \right).
		\end{aligned}
	\end{equation}
Hence by \eqref{hankel}, we have
$$
H_{2,1}(F_{f^{-1}}/2)=\cfrac{1}{48}\left(6c_1^4-6c_1^2c_2-3c_2^2+4 c_1c_3\right).
$$
Now using \eqref{c1}-\eqref{c3} and by straightforward computation,
\begin{equation}\label{3.2.3}
\begin{aligned}
H_{2,1}(F_{f^{-1}}/2)
&=\frac{13}{12}p_1^4-\frac{5}{2}(1-p_1^2)p_2-\frac{1}{12}(1-p_1^2)(3+p_1^2)p_2^2\\
&\,\,\,\, \quad \quad \quad +\frac{1}{3}(1-p_1^2)(1-|p_1^2|)p_1p_3.
\end{aligned}
\end{equation}
Now we have the following cases on $p_1$.\\[1mm]
\textbf{Case 1:} Let $p_1=1$, then by \eqref{3.2.3}, we obtain
$$
|H_{2,1}(F_{f^{-1}}/2)|=\frac{13}{12}.
$$
\textbf{Case 2:} Let $p_1=0$, then by \eqref{3.2.3}, we obtain
$$
|H_{2,1}(F_{f^{-1}}/2)|=\frac{1}{4}|p_2^2|\leq\frac{1}{4}.
$$
\textbf{Case 3:} Let $p_1 \in (0,1)$. Applying the triangle inequality in \eqref{3.2.3} and by using the fact
that $|p_3| \leq 1$, we obtain
\begin{equation}\label{3.2.4}
\begin{aligned}
H_{2,1}(F_{f^{-1}}/2)
&=\frac{1}{3}p_1(1-p_1^2)\left(\left| \frac{13p_1^3}{4(1-p_1^2}-\frac{5}{2}p_1 p_2-\frac{3+p_1^2}{4p_1}p_2^2 \right|+1-|p_2^2|\right)\\
&\leq \frac{1}{24}p_1(1-p_1^2)\left(\left| A+B p_2+C p_2^2 \right|+1-|p_2^2|\right),
\end{aligned}
\end{equation}
where 
$$
A:=\frac{13p_1^3}{4(1-p_1^2)}, \quad B:=-\frac{5}{2}p_1,\quad C:=-\frac{3+p_1^2}{4p_1}.
$$
Since $AC < 0$, so we can apply case (ii) of Lemma \ref{y(a,b,c)}.\\
[3mm]
\textbf{3(a).} Note that for $p_{1} \in(0,1)$ we have
$$
-4 A C\left(\frac{1}{C^{2}}-1\right)-B^{2}=-\cfrac{3p_{1}^{2}(16+p_1^2)}{(3+p_{1}^{2}} \leq  0.
$$
Moreover, the inequality $|B|<2(1-|C|)$ is equivalent to $3 - 4 p_1 + 2 p_1^2 < 0$ which is not true for $p_{1} \in(0,1)$.\\
[3mm]
\textbf{3(b).} It is easy to see that
$$
\min \left\{4(1+|C|)^{2},-4 A C\left(\frac{1}{C^{2}}-1\right)\right\}=-4 A C\left(\frac{1}{C^{2}}-1\right),
$$
and from $3(a),$ we know that 
$$
-4 A C\left(\frac{1}{C^{2}}-1\right) \leq B^{2}.
$$
Therefore, the inequality $B^{2} < \min \left\{4(1+|C|)^{2},-4 A C\left(\frac{1}{C^{2}}-1\right)\right\}$ does not holds for $0<p_1<1$.\\
[3mm]
\textbf{3(c).} Note that the inequality $|C|(|B|+4|A|)-|A B| \leq 0$ is equivalent to $44p_1^4-68p_1^2
-16-p_1^4\geq 0$, which is false for $p_{1} \in(0,1)$.\\
[3mm]
\textbf{3(d).}
Observe that the inequality 
$$|A B|-|C|(|B|-4|A|)=\cfrac{96 p_1^4+88 p_1^2-15}{1 - p1^2} \leq 0
$$
is equivalent to $96 p_1^4+88 p_1^2-15\leq 0$, which is true for 
$$
0<p_1\leq p_1''=\frac{1}{2}\sqrt{\cfrac{\sqrt{211}-11}{6}} \approx 0.38328.
$$
From \eqref{3.2.1} and Lemma \ref{y(a,b,c)}, we obtain
\begin{equation}
\begin{aligned}
H_{2,1}(F_{f^{-1}}/2)
&\leq \frac{1}{3}p_1(1-p_1^2)(-|A|+|B|+|C|)\\
&= \frac{1}{12}(3+8p_1^2-24p_1^4)=k(p_1)
\end{aligned}
\end{equation}
where $k(x)=3+8x^2-24x^4$, $ 0< x\leq p_1''$. Since the function $k'(x)>0$ in $ 0< x\leq p_1''$, we get $k(x)\leq k(p_1'')$ for $ 0< x\leq p_1''$.
Therefore, we have
$$
|H_{2,1}(F_{f^{-1}}/2)|\leq \cfrac{1}{48}(-58+5\sqrt{211)} \approx 0.304775.
$$
\textbf{3(e).} Furthermore, for $p_1''<p1<1$, from \eqref{3.2.3} and Lemma \ref{y(a,b,c)}, it follows that 
\begin{equation}
\begin{aligned}
H_{2,1}(F_{f^{-1}}/2)
&\leq \frac{1}{24}p_1(1-p_1^2)(|C|+|A|) \sqrt{1-\frac{B^{2}}{4 A C}}\\
&= \frac{1}{6}(12p_1^4 -2p_1^2+3)\sqrt{\cfrac{16-3p_1^2}{39+13p_1^2}}=k(p_1)
\end{aligned}
\end{equation}
where $$k(x)=\cfrac{1}{6}\sqrt{\cfrac{16-3x^2}{39+13x^2}}\left(12x^4 -2x^2+3\right) \text{for} p_1''<x<1.$$ 
\\[1mm]
As $k'(x)=0$ has no solution in $(p_1'',1)$ and $k'(x)>0$, the maximum exits at $x=1$.
Therefore we conclude that 
$$ |H_{2,1}(F_{f^{-1}}/2)| \leq k(1)=\cfrac{13}{12} \quad \text{for} \quad p_1''<x<1.
$$
Summarizing the cases 1-3 it follows that the inequality \eqref{thm2} holds. We now proceed to prove the equality part. Consider the Koebe function 
$$
k(z)=\frac{z}{(1-z)^2}.
$$
Clearly $k \in \mathcal{S}^*$ and it is easy to show that 
$$
|H_{2,1}(F_{k^{-1}}/2)|=\cfrac{13}{12}.
$$
This completes the proof.

\end{proof}

{\bf Acknowledgment.}
The first author thanks SERB-CRG and the second author's research work is supported by CSIR-UGC.

\end{document}